\documentclass{article}

\usepackage[margin=1in]{geometry}
\usepackage[parfill]{parskip}
\usepackage[utf8]{inputenc}
\usepackage{amsmath,amssymb,amsfonts,amsthm,mathrsfs,color}
\usepackage{graphicx}
\usepackage[skins]{tcolorbox}
\usepackage{bm}
\usepackage{multirow}
\usepackage{subcaption}
\usepackage[colorlinks=true,linkcolor=blue,citecolor=blue]{hyperref}

\newtheorem{Theorem}{Theorem}[section]
\newtheorem{Definition}[Theorem]{Definition}

\newtheorem{Remark}[Theorem]{Remark}  
\newtheorem{Assumption}[Theorem]{Assumption}

\newcommand\Tstrut{\rule{0pt}{2.6ex}}
\newcommand\Bstrut{\rule[-1.3ex]{0pt}{0pt}}

\makeatletter
\renewcommand{\fnum@figure}{Fig. \thefigure}
\makeatother

\allowdisplaybreaks

\newcommand\hl{\color{orange}\bf}

\let\OLDthebibliography\thebibliography
\renewcommand\thebibliography[1]{
  \OLDthebibliography{#1}
  \setlength{\parskip}{1pt}
  \setlength{\itemsep}{0pt plus 0.0ex}}

\title{Stein's method of moments for truncated multivariate distributions}
\author{Adrian Fischer\footnote{Adrian Fischer, University of Oxford, UK. E-mail: adrian.fischer@stats.ox.ac.uk}, Robert E. Gaunt\footnote{Robert E. Gaunt, The University of Manchester, UK. E-mail: robert.gaunt@manchester.ac.uk}\, and Yvik Swan\footnote{
Yvik Swan, Université libre de Bruxelles, Belgium. E-mail: yvik.swan@ulb.be}}

\begin{document}
\maketitle

\begin{abstract}
    We use Stein characterisations to derive new moment-type estimators for the parameters of several truncated multivariate distributions in the i.i.d.\ case; we also derive the asymptotic properties of these estimators. Our examples include the truncated multivariate normal distribution and truncated products of independent univariate distributions. The estimators are explicit and therefore provide an interesting alternative to the maximum-likelihood estimator (MLE). The quality of these estimators is assessed through competitive simulation studies, in which we compare their behaviour to the performance of the MLE and the score matching approach.
\end{abstract}


\section{Introduction}
Random variables are often only observed within a specific range; for example, due to technical boundaries of an experiment or geographical constraints. This necessitates methods to perform statistical inference on models with truncated probability distributions. The univariate case has been treated several times in the literature; see, for example, \cite{cohen1991truncated,schneider1986truncated}. Here, we want to focus rather on truncated multivariate probability distributions with the most prominent example probably being the truncated multivariate normal distribution, which is, for example, regularly utilised in censored and truncated regression models (see e.g.\ \cite{amemiya1974multivariate}) and for modelling the vector of drop-out prices observed in ascending auctions \cite{hong2003econometric}. \par

Despite this interest, there is little literature available on parameter estimation for general truncated multivariate probability distributions. In \cite{daskalakis2018efficient}, the authors propose an efficient algorithm for maximum likelihood estimation for the truncated multivariate normal distribution. The most natural competitor to the present work is \cite{liu2022estimating} in which the score matching approach is generalised to truncated multivariate probability distribution with only few assumptions on the truncation domain. \par

Our work is an extension of \cite{ebner2023point} in which the authors used the \textit{density approach} to Stein's method \cite{ley2017stein,ley2013stein} to obtain Stein operators for univariate distributions. Through this \emph{Stein's Method of Moments}, a new class of moment-type estimators is then retrieved by choosing appropriate test functions and replacing the expectation in the Stein identity by its empirical counterpart. Recently, an extension of the density approach to the multivariate paradigm has been developed in \cite{mijoule2023stein}. More precisely, let $X$ be a random vector with differentiable probability density function (pdf) $p_{\theta}$, which depends on an unknown parameter $\theta \in \Theta \subset \mathbb{R}^p$. Then, we have that
\begin{align} \label{Stein_equation_mult}
    \mathbb{E} \bigg[ \frac{\nabla \big( f(X) p_{\theta}(X) \big)}{p_{\theta}(X)} \bigg]=0
\end{align}
for all functions $f$ from a certain function class $\mathscr{F}_{\theta}$. We will call \eqref{Stein_equation_mult} resp.\ its empirical counterpart (expectation replaced by the sample mean for given observations) the \textit{Stein identity}. The differential operator
\begin{align} \label{def_mult_stein operator}
    \mathcal{A}f(x)= \frac{\nabla \big( f(x) p_{\theta}(x) \big)}{p_{\theta}(x)}
\end{align}
is then called a \textit{Stein operator} with respect to the probability distribution $p_{\theta}$. As it turns out, these operators are often of a simple form (note that a possibly non-tractable normalising constant vanishes in \eqref{def_mult_stein operator}) and therefore the empirical version of the Stein identity \eqref{Stein_equation_mult} can often be solved explicitly for $\theta$ resulting in an estimator for the latter. In the sequel, we will refer to an estimator obtained in a way as described above as a \textit{Stein estimator}. \par

The paper is organised as follows. In Section \ref{section_mult_trunc_normal}, we propose a new estimator for the truncated multivariate normal distribution with respect to any piecewise smooth truncation domain. In Section \ref{section_product_distr}, we consider truncated products of independent univariate distributions. Further, we investigate two product distributions more in detail: The product of a normal and a gamma as well as the product of a normal and a beta distribution. The performance of our proposed estimators is tested through competitive simulation studies.  \par

We briefly fix some notation. Let $\langle \cdot, \cdot \rangle$ be the standard scalar product and $\Vert \cdot \Vert$ be the Euclidean norm on $\mathbb{R}^d$. For a matrix $A \in \mathbb{R}^{d \times d}$, we define $\Vert A \Vert = \Vert \mathrm{vec} (A) \Vert$ and let $\otimes$ be the standard Kronecker product. We write $\frac{\partial f}{\partial x}$ for the partial derivative if $x$ is a scalar or for the matrix derivative of a (possibly vector or matrix-valued) function $f$ if $x$ is a vector or a matrix. When we differentiate a matrix-valued function with respect to a matrix-valued argument, we consider the vectorised function and the vectorised argument, i.e.
\begin{align*}
    \frac{\partial f}{\partial x} = \frac{\partial \mathrm{vec}(f)}{\partial \mathrm{vec}(x)}.
\end{align*}
Furthermore, we write $\nabla$ for the standard Jacobian of a vector-valued function with vector-valued argument. We denote by $B_r^d(x_0) = \{x \in \mathbb{R}^d \, \vert \, \Vert x - x_0 \Vert < r \}$ the open ball in $\mathbb{R}^d$ with radius $r>0$ and center $x_0$. For a subset $A \subset \mathbb{R}^d$ we will write $\overline{A}$ for its closure, $\mathrm{int}(A)$ for its interior, and  $\partial A= \overline{A} \setminus \mathrm{int}(A)$ for its boundary. We let $C(U,V)$ /  $C^k(U,V)$ / $C^{\infty}(U,V)$ be the sets of all continuous / $k$-times differentiable / smooth functions $f:U \rightarrow V$.

\section{Truncated multivariate normal distribution} \label{section_mult_trunc_normal}

The pdf of the truncated multivariate normal distribution $TN(\mu, \Sigma), \theta=(\mu,\Sigma)$  with $\mu \in \mathbb{R}^d$ and $\Sigma \in \mathbb{R}^{d \times d}$ positive definite, is given by
\begin{align*}
p_{\theta}(x)=\frac{1}{C(\theta)}\exp\bigg(-\frac{1}{2}(x-\mu)^{\top} \Sigma^{-1}(x-\mu) \bigg), \qquad x \in K,
\end{align*}
with normalising constant $C(\theta)$ and a truncation domain $K \subset \mathbb{R}^d$.
Let $A$ be a closed subset of $\mathbb{R}^d$ with non-empty interior. Then we write $M_A$ for the set of all points $p \in A$ such that there exists an open neighbourhood $V_p$ around $p$ in $\mathbb{R}^d$ so that $V_p \cap A$ is a $d$-dimensional smooth sub-manifold of $\mathbb{R}^d$. Let $\mathrm{Rd}(A)=\partial M_A$. Note that we have $\mathrm{Rd}(A) \subset \partial A$. We introduce the notion of a piecewise smooth domain as it will be needed for the technical assumptions on the truncation domain $K$ (see, for example, \cite[Example 3.2(d)]{amann2009analysis}). 
\begin{Definition}[Piecewise smooth domain]
Let $\mathcal{B}_{d-1}=(-1,1)^{d-1}$ be the open unit ball in $\mathbb{R}^{d-1}$ equipped with the maximum norm. A measurable subset $A$ of $\mathbb{R}^d$ with non-empty interior is called a piecewise smooth domain if there exist finitely many functions $h_j \in C(\overline{\mathcal{B}}_{d-1}, \mathbb{R}^{d}) \cap C^{\infty}(\mathcal{B}_{d-1}, \mathbb{R}^{d})$,  $j=1,\ldots, N$, such that
\begin{description}
\item[(a)] $h_j \vert_{\mathcal{B}_{d-1}}, 1 \leq j \leq N$ is a parametrisation of a subset of $\partial \overline{A}$,
\item[(b)] $\mathrm{Rd}(\overline{A})= \bigcup_{j=1}^N h_j(\mathcal{B}_{d-1})$,
\item[(c)] $\partial \overline{A}= \bigcup_{j=1}^N h_j(\overline{\mathcal{B}}_{d-1})$.
\end{description}
\end{Definition}
$\mathrm{Rd}(A)$ is the boundary $\partial A$ without singular points: If we take the unit cube $A=[-1,1]^d$ then $\mathrm{Rd}(A)$ equals $\partial A$ without all vertices of the cube. The functions $h_j$, $j=1,\ldots,N$, can then be chosen such that each $h_j$ parametrises one side of the cube. 
\begin{Assumption} \label{ass_piece_smooth_domain}
$K \cap B_r^d(x_0)$ is a piecewise smooth domain in $\mathbb{R}^d$ for some $x_0 \in K$ for all $r>0$.
\end{Assumption}
We use the density approach operator, which is given by
\begin{align*}
\mathcal{A}_{\theta}f(x)=   \nabla f(x) + \frac{\nabla p(x)}{p(x)} f(x) =  \nabla f(x) -  \Sigma^{-1} (x-\mu)  f(x), \qquad x \in \mathrm{int}(K),
\end{align*}
for a differentiable functions $f:\overline{K} \rightarrow \mathbb{R}$ (compare to \cite[Definition 3.17]{mijoule2023stein}). We define the Stein operator for a vector-valued function $f:\mathbb{R}^d \rightarrow \mathbb{R}^d$ by applying $\mathcal{A}_{\theta}$ to each component of $f$, i.e.
\begin{align*}
\mathcal{A}_{\theta}f(x)=  \nabla f(x)^{\top} -  \Sigma^{-1} (x-\mu)  f(x)^{\top} \in \mathbb{R}^{d \times d} , \qquad x \in \mathrm{int}(K).
\end{align*}
Moreover, we introduce the class of functions
\begin{align*}
\mathscr{F}_{\theta}=\bigg\{ f: \overline{K} \rightarrow \mathbb{R} \, \vert & \, f \in C^{\infty}(\mathrm{int}(K),\mathbb{R}) \cap C(\overline{K},\mathbb{R}), \, f=0 \text{ on } \partial \overline{K} \text{ and } \\
& \lim_{\Vert x \Vert \rightarrow \infty} f(x) p_{\theta}(x) \Vert x \Vert^{d-1}=0, \, \int_K \big\Vert \nabla \big( f(x) p_{\theta}(x) \big) \big\Vert\, dx < \infty \bigg\},
\end{align*}
and for vector-valued functions respectively. The condition for $\Vert x \Vert \rightarrow \infty$ is only necessary if $K$ is unbounded. We then let $\mathscr{F}= \cap_{\theta \in \Theta} \mathscr{F}_{\theta}$.
\begin{Theorem} \label{theorem_trunc_norm_stein_op}
Suppose Assumption \ref{ass_piece_smooth_domain} holds. Then, for any scalar- or vector-valued function $f \in \mathscr{F}$ and $X \sim TN(\mu, \Sigma)$ we have $\mathbb{E}[\mathcal{A}_{\theta}f(X)]=0$.
\end{Theorem}
\begin{proof}
We prove the result for a function $f=(f^{(1)},\ldots,f^{(d)}):\mathbb{R}^d \rightarrow \mathbb{R}^d$. Let $\tilde{f}_{i,j}=(\tilde{f}_{i,j}^{(1)},\ldots,\tilde{f}_{i,j}^{(d)}):\mathbb{R}^d \rightarrow \mathbb{R}^d$, $1 \leq i,j \leq d$, be functions such that
\begin{align*}
\tilde{f}_{i,j}^{(k)}(x)=
\begin{cases}
f^{(j)}(x) & k=i \\
0 & \text{otherwise}
\end{cases}.
\end{align*}
With dominated convergence we have
\begin{align*}
\mathbb{E}[\mathcal{A}_{\theta} f(X)]_{i,j} &= \int_{K} \mathrm{div} \big( \tilde{f}_{i,j}(x) p_{\theta}(x) \big) \,dx  \\
&= \lim_{r \rightarrow \infty} \int_{K\cap B_r^d(x_0)} \mathrm{div} \big( \tilde{f}_{i,j}(x) p_{\theta}(x) \big)\, dx 
\end{align*}
for all $i,j, x_0 \in K$. Then the divergence theorem (see e.g.\ \cite[Theorem XII.3.11]{amann2009analysis}) ensures that for all $r>0$ we have
\begin{align*}
&\bigg\vert \int_{K\cap B_r^d(x_0)} \mathrm{div} \big( \tilde{f}_{i,j}(x) p_{\theta}(x) \big) \,dx \bigg\vert = \bigg\vert \int_{\mathrm{Rd} (\overline{K} \cap \overline{B_r^d(x_0)})} p_{\theta}(x) \langle \tilde{f}_{i,j}(x) , \vec{n}(x) \rangle \,d\sigma(x) \bigg\vert \\
\leq & \int_{\mathrm{Rd}( \overline{K})} \big\vert p_{\theta}(x) \langle \tilde{f}_{i,j}(x), \vec{n}(x) \rangle \big\vert \,d\sigma(x) + \int_{\partial \overline{B_r^d(x_0)}} \big\vert p_{\theta}(x) \langle \tilde{f}_{i,j}(x), \vec{n}(x) \rangle \big\vert \,d\sigma(x),
\end{align*} 
where $d\sigma(x)$ denotes integration with respect to the surface measure and $\vec{n}(x)$ is the outward pointing unit vector orthogonal to the surface at $x$. For the second integral, we set the integrand equal to $0$ if $x \notin K$. The first integral is equal to zero; for the second integral we have by using the spherical parametrisation of $\partial \overline{B_r^d(x_0)}$ and dominated convergence that
\begin{align*}
& \lim_{r \rightarrow \infty} \int_{\partial \overline{B_r^d(x_0)}} \big\vert p_{\theta}(x) \langle \tilde{f}_{i,j}(x), \vec{n}(x) \rangle \big\vert \,d\sigma(x) \\
 = &\int_{0}^{2\pi} \int_{0}^{\pi} \ldots \int_{0}^{\pi} \lim_{r \rightarrow \infty} \big\vert p_{\theta}(r,\varphi) f^{(j)}(r,\varphi) \vec{n}_i(\varphi) \big\vert r^{d-1} \sin^{d-2}(\varphi_1) \cdots \sin(\varphi_{d-2}) \,d\varphi_1\ldots \,d\varphi_{d-1} 
\end{align*}
for $\varphi=(\varphi_1,\ldots, \varphi_{d-1})$ and $\vec{n}(x)=(\vec{n}_1(x),\ldots,\vec{n}_d(x))$, whereby the latter vector is independent of $r=\Vert x \Vert$. Now $r \rightarrow \infty$ implies $\Vert x \Vert \rightarrow \infty$ and we conclude that the latter expression is equal to $0$.
\end{proof}
Let $X_1,\ldots,X_n \sim TN(\mu_0,\Sigma_0)$ be an i.i.d.\ sample living on a common probability space $(\Omega,\mathcal{F},\mathbb{P})$. For one scalar-valued $f_1 \in \mathscr{F}$ and one $\mathbb{R}^d$-valued test function $f_2 \in \mathscr{F}$ we solve the system of equations
\begin{align*}
\frac{1}{n}\sum_{i=1}^n \mathcal{A}_{\theta}f_1(X_i) = 0, \qquad \frac{1}{n}\sum_{i=1}^n \mathcal{A}_{\theta}f_2(X_i) = 0
\end{align*}
for $\theta=(\mu,\Sigma)$ and arrive at the Stein estimators
\begin{align*}
\hat{\Sigma}_n&=\frac{1}{2}\big(\tilde{\Sigma}_n + \tilde{\Sigma}_n^{\top}\big), \\
\hat{\mu}_n &= \frac{\overline{ X f_1(X)} - \hat{\Sigma}_n  \overline{ \nabla f_1(X)}}{ \overline{ f_1(X)}},
\end{align*}
where
\begin{align*}
\tilde{\Sigma}_n= \big( \overline{X f_2(X)^{\top}} \,  \overline{f_1(X)} -  \overline{X f_1(X)} \,  \overline{f_2(X)}^{\top} \big) \big( \overline{\nabla f_2(X)}^{\top} \,  \overline{f_1(X)} -  \overline{ \nabla f_1(X)} \,  \overline{f_2(X)}^{\top} \big)^{-1}.
\end{align*}
In the display above we wrote $\overline{f(X)}=\frac{1}{n}\sum_{i=1}^n f(X_i)$ for a function $f \in \mathscr{F}$. Note that we symmetrised the matrix $\tilde{\Sigma}_n$ as it is not necessarily symmetric. However, it is still possible that $\hat{\Sigma}_n$ is not positive-definite. The possibility for the estimate to lie outside of the parameter space is a known issue for moment-type estimators. Here and in the next section, we will write $\hat{\theta}_n$ for a (family of) Stein estimator(s). For the truncated multivariate normal we therefore have $\hat{\theta}_n=(\hat{\mu}_n,\hat{\Sigma}_n)$. \par
In the next theorem, we provide conditions on the test functions under which the proposed estimators exist and are consistent. In this regard, we introduce a new set of assumptions.
\begin{Assumption}\label{ass_trunc_normal_expect}
For $f_1,f_2 \in \mathscr{F}$ (where $f_1$ is scalar- and $f_2$ is vector-valued) we have that
\begin{align*}
\mathbb{E}[ \Vert X f_2(X)^{\top} \Vert],\: \mathbb{E}[\vert f_1(X) \vert],\:\mathbb{E}[\Vert X f_1(X) \Vert] ,\:  \mathbb{E}[\Vert f_2(X)\Vert ], \:\mathbb{E}[ \Vert \nabla f_2(X) \Vert],\: \mathbb{E}[\Vert \nabla f_1(X) \Vert] < \infty,
\end{align*}
$ \mathbb{E}[\nabla f_2(X)^{\top} f_1(X)] - \mathbb{E}[ \nabla f_1(X) f_2(X)^{\top}]$ is non-singular, and $ \mathbb{E}[f_1(X)]\neq 0$ for $X \sim TN(\mu_0,\Sigma_0)$.
\end{Assumption}

We introduce the function
\begin{align*}
G: \mathbb{R}^{d \times d} \times \mathbb{R}^{d \times d} \times \mathbb{R}^{d} \times \mathbb{R}^{d} \times \mathbb{R}^{d} \times \mathbb{R} \supset \widetilde{D} \rightarrow \mathbb{R}^{d \times d} \times \mathbb{R}^{d}
\end{align*}
defined through
\begin{align*}
G(Z) = 
\begin{pmatrix}
G_1(Z) \\
G_2(Z)
\end{pmatrix}
=
\begin{pmatrix}
(Z_1 z - z_1 z_2^{\top} ) (Z_2z - z_3 z_2^{\top})^{-1} \\
\frac{1}{z} (z_1 - G_1(Z) z_3)
\end{pmatrix},
\end{align*}
where $Z=(Z_1,Z_2,z_1,z_2,z_3,z)$ and $\widetilde{D} $ contains all $(Z_1,Z_2,z_1,z_2,z_3,z)$ such that $(Z_2z - z_3 z_2^{\top})$ is invertible and $z \neq 0$. Note that $\widetilde{D} $ is an open set.

\begin{Theorem} \label{theorem_truncnorm_consistency}
Suppose that Assumptions \ref{ass_piece_smooth_domain} and \ref{ass_trunc_normal_expect} hold. Then $(\hat{\Sigma}_n,\hat{\mu}_n)$ exist with probability converging to one and are strongly consistent in the following sense: There is a set $A \subset \Omega$ with $\mathbb{P}(A)=1$ such that for each $\omega \in A$ there is a $N \in \mathbb{N}$ such that $(\hat{\Sigma}_n,\hat{\mu}_n)$ exist for each $n \geq N$ and 
\begin{align*}
\hat{\theta}_n(\omega) \overset{\mathrm{a.s.}}{\longrightarrow} \theta_0
\end{align*}
as $n \rightarrow \infty$.
\end{Theorem}
\begin{proof}
Note first that the second part of Assumption \ref{ass_trunc_normal_expect} and Theorem \ref{theorem_trunc_norm_stein_op} entails that
\begin{align*}
\Sigma&= \big( \mathbb{E}[X f_2(X)^{\top}  \mathbb{E}[f_1(X)] -  \mathbb{E}[X f_1(X)] \mathbb{E}[f_2(X)]^{\top} \big) \big( \mathbb{E}[\nabla f_2(X)]^{\top} \mathbb{E}[f_1(X)] -  \mathbb{E}[\nabla f_1(X)]  \mathbb{E}[f_2(X)]^{\top} \big)^{-1}, \\
\mu &= \frac{\mathbb{E}[ X f_1(X)] - \Sigma  \mathbb{E}[\nabla f_1(X)]}{\mathbb{E}[f_1(X)]},
\end{align*}
for $X \sim TN(\mu,\Sigma)$. By the strong law of large numbers and Assumption \ref{ass_trunc_normal_expect}, it is clear that 
\begin{align*}
\overline{X f_2(X)^{\top}},\:  \overline{f_1(X)},\:  \overline{X f_1(X)},\:  \overline{f_2(X)},\: \overline{\nabla f_2(X)},\:\overline{ \nabla f_1(X)}
\end{align*}
converge almost surely to their respective expectations. Let $V \subset \mathbb{S}^{d \times d} \times \mathbb{R}^d$ (where $\mathbb{S}^{d \times d}$ denotes the set of all symmetric matrices) be open with $(\Sigma_0,\mu_0) \in V$. Then with the continuity of the function $G$ we know that there exists an open set $U \subset \widetilde{D} $ with
\begin{align*}
\big(\mathbb{E}[ X f_2(X)^{\top} ],\:\mathbb{E}[ \nabla f_2(X)],\:\mathbb{E}[ X f_1(X) ] ,\:  \mathbb{E}[ f_2(X) ],\: \mathbb{E}[ \nabla f_1(X) ] ,\:\mathbb{E}[ f_1(X)]\big)^{\top} \in U, X \sim TN(\mu_0,\Sigma_0) 
\end{align*}
such that $\widetilde{G} \circ G(U) \subset V$, where
\begin{align*}
\widetilde{G}:\mathbb{R}^{d \times d} \times \mathbb{R}^d \rightarrow \mathbb{R}^{d \times d} \times \mathbb{R}^d, (Z,z) \mapsto \bigg(\frac{1}{2}\big(Z+Z^{\top}\big),z\bigg).
\end{align*}
Note that the set of all positive definite matrices is open within the set of symmetric matrices. Then with
\begin{align*}
A_n= \Big\{ \Big( \overline{X f_2(X)^{\top}},\: \overline{\nabla f_2(X)},\:  \overline{X f_1(X)},\:  \overline{f_2(X)},\:\overline{ \nabla f_1(X)},\:\overline{f_1(X)}\Big)^{\top} \in U \Big\}
\end{align*}
we have that $\mathbb{P}(A_n) \rightarrow 1$ as $n \rightarrow \infty$ and the consistency part follows by the continuous mapping theorem.
\end{proof}
We now show that our estimators are asymptotically normal and calculate the asymptotic covariance matrix. For the latter purpose we need the derivatives of $G_1$ and calculate
\begingroup
\allowdisplaybreaks
\begin{align*}
\frac{\partial G_1(Z)}{\partial Z_1}=& (Z_2z-z_3z_2^{\top})^{-\top} \otimes I_d, \\
\frac{\partial G_1(Z)}{\partial Z_2}=& - (Z_2z-z_3z_2^{\top})^{-\top}  \otimes G_1(Z) ,  \\
\frac{\partial G_1(Z)}{\partial z_1}=& -  \big( (Z_2z-z_3z_2^{\top})^{-\top}z_2 \big) \otimes I_d,  \\
\frac{\partial G_1(Z)}{\partial z_2}=& -   (Z_2z-z_3z_2^{\top})^{-\top} \otimes (G_1(Z)z_3+z_1) , \\
\frac{\partial G_1(Z)}{\partial z_3}=&   \big( (Z_2z-z_3z_2^{\top})^{-\top}z_2 \big)\otimes G_1(Z) , \\
\frac{\partial G_1(Z)}{\partial z}=& \big( (Z_2z-z_3z_2^{\top})^{-\top} \otimes I_d \big)\mathrm{vec}(Z_1) - \big(  (Z_2z-z_3z_2^{\top})^{-\top} \otimes G_1(Z) \big)\mathrm{vec}(Z_2). 
\end{align*}
\endgroup
In the same manner, we obtain for $G_2$ that
\begin{gather*}
\frac{\partial G_2(Z)}{\partial Z_1}= -\Big( \frac{1}{z}z_3^{\top} \otimes I_d \Big)\frac{\partial G_1(Z)}{\partial Z_1}, \qquad
\frac{\partial G_2(Z)}{\partial Z_2}= -\Big( \frac{1}{z}z_3^{\top} \otimes I_d \Big)\frac{\partial G_1(Z)}{\partial Z_2}, \\ 
\frac{\partial G_2(Z)}{\partial z_1}= \frac{1}{z}I_d- \Big( \frac{1}{z}z_3^{\top} \otimes I_d \Big) \frac{\partial G_1(Z)}{\partial z_1}, \qquad \frac{\partial G_2(Z)}{\partial z_2}= - \Big( \frac{1}{z}z_3^{\top} \otimes I_d \Big)  \frac{\partial G_2(Z)}{\partial z_2}, \\
\frac{\partial G_2(Z)}{\partial z_3}= - \Big( \frac{1}{z}z_3^{\top} \otimes I_d \Big)\frac{\partial G_1(Z)}{\partial z_3} - \frac{1}{z}G_1(Z), \\
\frac{\partial G_2(Z)}{\partial z}= -\frac{1}{z^2}(z_1-G_1(Z)z_3) - \Big( \frac{1}{z}z_3^{\top} \otimes I_d\Big)\frac{\partial G_1(Z)}{\partial z}.
\end{gather*}
If we now define $\widetilde{G}_1(Z)=\frac{1}{2}(G_1(Z)+G_1(Z)^{\top})$ we have that
\begin{align*}
\frac{\partial \widetilde{G}_1(Z)}{\partial Z}= \frac{1}{2}\bigg(\frac{ \partial G_1(Z)}{\partial Z} + \mathcal{K}_{d,d} \frac{ \partial G_1(Z)}{\partial Z}\bigg),
\end{align*}
where $\mathcal{K}_{p,q}$ is the commutation matrix and
\begin{align*}
\frac{\partial \widetilde{G}_1(Z)}{\partial Z}=
\begin{pmatrix} \displaystyle
\frac{\partial \widetilde{G}_1(Z)}{\partial Z_1} & \displaystyle\frac{\partial \widetilde{G}_1(Z)}{\partial Z_2} &\displaystyle \frac{\partial \widetilde{G}_1(Z)}{\partial z_1} &\displaystyle \frac{\partial \widetilde{G}_1(Z)}{\partial z_2} & \displaystyle\frac{\partial \widetilde{G}_1(Z)}{\partial z_3} &\displaystyle \frac{\partial \widetilde{G}_1(Z)}{\partial z} 
\end{pmatrix}
\end{align*}
and respectively for $G_2$. Hence, with $\widetilde{G}(Z)=(\widetilde{G}_1(Z),G_2(Z))^{\top}$ we arrive at
\begin{align*}
\frac{\partial \widetilde{G}(Z)}{\partial Z}=
\begin{pmatrix}
\frac{\partial \widetilde{G}_1(Z)}{\partial Z_1} & \frac{\partial \widetilde{G}_1(Z)}{\partial Z_2} & \frac{\partial \widetilde{G}_1(Z)}{\partial z_1} & \frac{\partial \widetilde{G}_1(Z)}{\partial z_2} & \frac{\partial \widetilde{G}_1(Z)}{\partial z_3} & \frac{\partial \widetilde{G}_1(Z)}{\partial z} \\[5pt]
 \frac{\partial G_2(Z)}{\partial Z_1} &  \frac{\partial G_2(Z)}{\partial Z_2} &  \frac{\partial G_2(Z)}{\partial z_1} &  \frac{\partial G_2(Z)}{\partial z_2} &  \frac{\partial G_2(Z)}{\partial z_3} &  \frac{\partial G_2(Z)}{\partial Z}  
\end{pmatrix}
\in \mathbb{R}^{(d^2+d) \times (2d^2+3d+1)}.
\end{align*}
It is clear by the multivariate central limit theorem that the sequence of random vectors defined by
\begin{align*}
Y_n = 
\begin{pmatrix}
\mathrm{vec}\Big(\overline{Xf_2(X)^{\top}}\Big)^{\top} & \mathrm{vec}\Big(\overline{\nabla f_2(X)}^{\top}\Big)^{\top} & \overline{X f_1(X)}^{\top} & \overline{f_2(X)}^{\top} & \overline{\nabla f_1(X)}^{\top} & \overline{f_1(X)}
\end{pmatrix}^{\top},
\end{align*}
$X \sim TN(\mu_0,\Sigma_0)$, is asymptotically normal, i.e.
\begin{align*}
\sqrt{n}(Y_n - \mathbb{E}[Y_1]) \overset{D}{\longrightarrow} N(0,\mathrm{Var}[Y_1]),
\end{align*}
as $n\rightarrow\infty$, if the covariance matrix $\mathrm{Var}[Y_1]$ exists and is invertible. Then the multivariate delta method yields the asymptotic normality of the Stein estimator $\hat{\theta}_n $ and we have proved the following theorem. Note that we have $\widetilde{G}(\mathbb{E}[Y_1])=(\Sigma,\mu)$ (where we embedded $\mathbb{E}[Y_1]$ appropriately into the domain of $\widetilde{G}$).
\begin{Theorem}
Suppose Assumptions \ref{ass_piece_smooth_domain} and \ref{ass_trunc_normal_expect} hold and that $\mathrm{Var}[Y_1]$ exists and is invertible, where $Y_1$ is defined as above. Then the Stein estimator $(\hat{\Sigma}_n,\hat{\mu}_n) $ is asymptotically normal, i.e.
\begin{align*}
\sqrt{n} \bigg( 
\begin{pmatrix}
\mathrm{vec}(\hat{\Sigma}_n) \\
\hat{\mu}_n
\end{pmatrix}
- 
\begin{pmatrix}
\mathrm{vec}(\Sigma_0) \\
\mu_0
\end{pmatrix}
\bigg)
\overset{D}{\longrightarrow} N\bigg(0, \bigg(\frac{\partial \widetilde{G}(Z)}{\partial Z}\Big\vert_{\mathbb{E}[Y_1]}\bigg) \mathrm{Var}[Y_1] \bigg(\frac{\partial \widetilde{G}(Z)}{\partial Z}\Big\vert_{\mathbb{E}[Y_1]} \bigg)^{\top} \bigg),
\end{align*}
as $n \rightarrow \infty$, where all quantities in the formula for the asymptotic covariance matrix have been defined above.
\end{Theorem}
In the sequel, we tackle the question of how to choose appropriate test functions. It is easy to see that in the untruncated case ($K=\mathbb{R}^d$) the functions $f_1: \mathbb{R}^d \rightarrow \mathbb{R}$, $x \mapsto 1$ and $f_2: \mathbb{R}^d \rightarrow \mathbb{R}^{d}$, $x \mapsto x$ yield the maximum likelihood estimator (MLE)
\begin{align*}
\hat{\Sigma}_n = \frac{1}{n} \sum_{i=1}^n (X_i - \overline{X})(X_i - \overline{X})^{\top}, \qquad \hat{\mu}_n=\overline{X}.
\end{align*}
\begin{Remark}
Another Stein operator for the standard multivariate normal $N(\mu, \Sigma)$ distribution is given by
\begin{align*}
\mathcal{A}_{\theta}f(x)= (x-\mu)^{\top} \nabla f(x) - \nabla^{\top} \Sigma \nabla f(x), \qquad x \in \mathbb{R}^d,
\end{align*}
for functions $f:\mathbb{R}^d \rightarrow \mathbb{R}$ (see \cite{barbour1990stein,gotze1991rate}). If we choose the functions $f_1: \mathbb{R}^d \rightarrow \mathbb{R}^d$, $x \mapsto x$ and $f_2: \mathbb{R}^d \rightarrow \mathbb{R}^{d \times d}$,  $x \mapsto x x^{\top}$, we obtain the MLE. For that, one has to apply the functions value-wise and note that
$
n^{-1} \sum_{i=1}^n (X_i - \overline{X})(X_i - \overline{X})^{\top}= n^{-1} \sum_{i=1}^n  (X_i - \overline{X}) X_i.
$
\end{Remark} 
Now suppose that there exist (on $\mathrm{int}(K)$) differentiable functions $\kappa_i: \overline{K} \rightarrow \mathbb{R}$, 
 $i=1,\ldots,I$, with
\begin{align} \label{funcion_kappa_definition}
\partial\overline{K} \subset \bigcup_{i=1}^I \{ x \in \overline{K} \, \vert \, \kappa_i(x)=0 \}.
\end{align}
The latter definition includes, for example, any $d$-dimensional ellipse but also sets whose boundaries are non-differentiable curves such as cuboids. For the latter, let
\begin{align*}
K=(a_1,b_1) \times \ldots \times (a_d,b_d),
\end{align*}
and we can define $2d$ functions given by $\kappa_1(x)=x_1-a_1$, $\kappa_2(x)=x_1-b_1,\ldots,$ $\kappa_{2d-1}(x)=x_d-a_d$, $\kappa_{2d}(x)=x_d-b_d$. Furthermore, we let
\begin{align*}
\kappa(x)=\prod_{i=1}^I \kappa_i(x).
\end{align*}
Motivated by the test functions that yield the MLE in the untruncated case we propose
\begin{align*}
f_1(x)=\kappa(x), \qquad f_2(x)=x\kappa(x),
\end{align*}
and denote the corresponding estimators by $\hat{\theta}_n^{\mathrm{ST}}=(\hat{\mu}_n^{\mathrm{ST}},\hat{\Sigma}_n^{\mathrm{ST}} )$. Note that indeed $f_1$ is scalar- and $f_2$ is vector-valued. One still has to make sure that a chosen test function belongs to the corresponding function class and that Assumption \ref{ass_trunc_normal_expect} is satisfied. In fact, for any truncation domain $K$, one could pick $\kappa(x)=0$ in \eqref{funcion_kappa_definition} which would yield $f_1(x)=0$ as well as $f_2(x)=0$ for which Assumption \ref{ass_trunc_normal_expect} is clearly not satisfied. However, we still want to allow $\kappa(x)\neq 0$  if $x \notin \partial \overline{K}$ to add some flexibility in order to choose a suitable function $\kappa$. One might think of the case where $K=\cup_{i=1}^I B_{a}^d(ia e_k)$ for some $a>0$ and $e_k=(e_k^{(1)},\ldots,e_k^{(d)})$ the $k$th unit vector in $\mathbb{R}^d$, where we can simply choose $\kappa(x)=\prod_{i=1}^I \kappa_i(x)$ and $\kappa_i(x)= \sum_{j=1}^d (x_j-e_k^{(j)})^2-a^2$. \par

We performed a competitive simulation study whose results can be found in Table \ref{mult_trunc_normal_sim}. The study was performed for $d=2$ and with respect to the rectangular truncation domain $K=(-1,1) \times (-1,1)$. We compared the Stein estimator $\hat{\theta}_n^{\mathrm{ST}}$ to the MLE $\hat{\theta}_n^{\mathrm{ML}}$ and the score matching approach $\hat{\theta}_n^{\mathrm{SM}}$ from \cite{liu2022estimating}. 

For the MLE, we numerically calculated the maximum of the log-likelihood function
$
    \theta \mapsto \sum_{i=1}^n \log p_{\theta}(X_i).
$
Therefore, we parametrised the positive-definite and symmetric covariance matrix $\Sigma$ through the Cholesky decomposition $\Sigma=L L^{\top}$, where $L$ is a lower triangular matrix and therefore possesses $d(d+1)/2$ elements. Optimisation is then performed with respect to $\theta=(\mu,LL^{\top})$ and includes $d(d+1)/2+d$ variables, which ensures that the resulting estimator for $\Sigma$ is positive definite. Note that MLE involves the calculation of the normalising constant $C(\theta)$, which is performed via numerical integration; we used the R package \textit{cubature} \cite{cubature}. The numerical optimisation for the score matching estimator $\hat{\theta}_n^{\mathrm{SM}}$ was performed in the same way, whereby a computation of the normalising constant is not necessary. For the Stein estimator, we chose $\kappa(x)=(x_1-1)(x_1+1)(x_2-1)(x_2+1)$ and $f_1(x)=\kappa(x)$, $ f_2(x)=x\kappa(x)$, as proposed in the preceding paragraph. In order to evaluate the performance of our proposed estimators, we calculated the mean squared error (MSE) for both parameters. As per $\mu$, MSE stands for the average Euclidean distance between estimated and true value, that is the sample mean of $\Vert \mu_0-\hat{\mu}_n^{\bullet} \Vert$ with respect to all iterations of the simulation, where $\bullet=\mathrm{ML}$ or $\bullet=\mathrm{ST}$. Regarding $\Sigma$, we used the average spectral norm to measure the distance, i.e.\ the sample mean of $\Vert \Sigma_0-\hat{\Sigma}_n^{\bullet} \Vert$ with respect to all iterations of the simulation. It is worth noting that for higher dimensions ($d\geq 3$), numerical optimisation for the MLE becomes tedious with very slow convergence or no convergence at all, which is not surprising as the dimension of the parameter space grows quadratically with $d$. Instead, $\hat{\theta}_n^{\mathrm{ST}}$ seemed to give reliable results for non-extreme parameter values and an adequate sample size (see the parameter constellations and sample size chosen in the simulation study). For the purpose of a proper comparison we restricted ourselves to the two-dimensional case. 

As can be seen in Table \ref{mult_trunc_normal_sim}, even for $d=2$ the MLE and the score matching estimator seem to break down completely for certain parameter constellations while the Stein estimator is still reliable. Otherwise, all three estimators perform similarly whereby we emphasise that $\hat{\theta}_n^{\mathrm{ML}}$ and $\hat{\theta}_n^{\mathrm{SM}}$ require a complicated numerical procedure while $\hat{\theta}_n^{\mathrm{ST}}$ is completely explicit and easy to calculate. An estimation result is considered as not eligible if the algorithm threw an error or if the estimation result lies outside of the parameter space (for this example this is the case when the estimated covariance matrix is not positive definite). We added a column \textit{NE} to the table which reports the estimated number of cases (out of $100$) where an estimator is not eligible. There were no problems in this regard as it can be observed in the corresponding column which is in line with the rather large sample size chosen. Also, for a complicated truncation domain, the calculation of $C(\theta)$ might not be tractable anymore since it needs to be done numerically. However, for the Stein estimator it suffices to possess a function $\kappa$ as explained in this section which describes the boundary of the truncation domain. We also refer to \cite{amemiya1974multivariate} and \cite{lee1979first} in which an explicit estimator of the one-sided truncated multivariate normal distribution (meaning that each component of the random vector is truncated with respect to one side) is discussed. However, we did not include these estimators in our simulation study since one is limited regarding the choice of a truncation domain. 

\begin{table} 
\centering
\begin{tabular}{cc|ccc|ccc}
 $(\mu_0,\Sigma_0)$ & & \multicolumn{3}{c}{MSE} & \multicolumn{3}{|c}{NE} \\  \hline
 & & $\hat{\theta}_n^{\mathrm{ML}}$ & $\hat{\theta}_n^{\mathrm{SM}}$ &  $\hat{\theta}_n^{\mathrm{ST}}$ & $\hat{\theta}_n^{\mathrm{ML}}$ & $\hat{\theta}_n^{\mathrm{SM}}$ &  $\hat{\theta}_n^{\mathrm{ST}}$  \\ \hline
\multirow{2}{*}{$\bigg(\begin{pmatrix} 0 \\ 0 \end{pmatrix},\begin{pmatrix} 1 & 0 \\ 0 & 1 \end{pmatrix}\bigg)$} & $\mu$\Tstrut  & \hl $\hl0.078$ & $0.12$ & $0.085$ & \multirow{2}{*}{$0$} & \multirow{2}{*}{$0$} & \multirow{2}{*}{$0$} \\ & $\Sigma$\Bstrut  & \hl $\hl 0.346$ & $1\text{e5}$ & $0.393$ \\ \hline 
\multirow{2}{*}{$\bigg(\begin{pmatrix} 0.5 \\ 0.5 \end{pmatrix},\begin{pmatrix} 1 & 0 \\ 0 & 1 \end{pmatrix}\bigg)$} & $\mu$\Tstrut  & \hl $\hl 0.18$ & $0.191$ & $0.197$ & \multirow{2}{*}{$0$} & \multirow{2}{*}{$0$} & \multirow{2}{*}{$0$} \\ & $\Sigma$\Bstrut  & \hl $\hl 0.369$ & $0.396$ & $0.408$ \\ \hline 
\multirow{2}{*}{$\bigg(\begin{pmatrix} 0.5 \\ 0.5 \end{pmatrix},\begin{pmatrix} 0.5 & 0 \\ 0 & 0.5 \end{pmatrix}\bigg)$} & $\mu$\Tstrut  & \hl $\hl 0.077$ & $0.082$ & $0.084$ & \multirow{2}{*}{$0$} & \multirow{2}{*}{$0$} & \multirow{2}{*}{$0$} \\ & $\Sigma$\Bstrut  & \hl $\hl 0.09$ & $0.098$ & $0.099$ \\ \hline 
\multirow{2}{*}{$\bigg(\begin{pmatrix} 0 \\ 0 \end{pmatrix},\begin{pmatrix} 2 & 0 \\ 0 & 2 \end{pmatrix}\bigg)$} & $\mu$\Tstrut  & $6.88$ & $487$ & \hl $\hl 0.286$ & \multirow{2}{*}{$0$} & \multirow{2}{*}{$0$} & \multirow{2}{*}{$0$} \\ & $\Sigma$\Bstrut  & $2.2\text{e6}$ & $9.38\text{e9}$ & \hl $\hl 3.81$ \\ \hline 
\multirow{2}{*}{$\bigg(\begin{pmatrix} 0 \\ 0 \end{pmatrix},\begin{pmatrix} 0.2 & 0 \\ 0 & 0.2 \end{pmatrix}\bigg)$} & $\mu$\Tstrut  & $0.025$ & $0.042$ & \hl $\hl 0.021$ & \multirow{2}{*}{$0$} & \multirow{2}{*}{$0$} & \multirow{2}{*}{$0$} \\ & $\Sigma$\Bstrut  & $0.026$ & $0.045$ & \hl $ \hl 0.019$ \\ \hline 
\multirow{2}{*}{$\bigg(\begin{pmatrix} 0.8 \\ -0.2 \end{pmatrix},\begin{pmatrix} 0.5 & 0 \\ 0 & 0.5 \end{pmatrix}\bigg)$} & $\mu$\Tstrut  & \hl $\hl 0.092$ & $0.098$ & $0.099$ & \multirow{2}{*}{$0$} & \multirow{2}{*}{$0$} & \multirow{2}{*}{$0$} \\ & $\Sigma$\Bstrut  & \hl $\hl 0.094$ & $0.101$ & $0.103$ \\ \hline 
\multirow{2}{*}{$\bigg(\begin{pmatrix} 0 \\ 0 \end{pmatrix},\begin{pmatrix} 5 & 0.4 \\ 0.4 & 5 \end{pmatrix}\bigg)$} & $\mu$\Tstrut  & \hl $\hl 0.042$ & $101$ & $0.044$ & \multirow{2}{*}{$0$} & \multirow{2}{*}{$0$} & \multirow{2}{*}{$0$} \\ & $\Sigma$\Bstrut  & $0.129$ & $4.65\text{e7}$ & \hl $\hl 0.128$ \\ \hline 
\multirow{2}{*}{$\bigg(\begin{pmatrix} 0 \\ 0 \end{pmatrix},\begin{pmatrix} 0.8 & -0.7 \\ -0.7 & 0.9 \end{pmatrix}\bigg)$} & $\mu$\Tstrut & $128$ & $119$ & \hl $\hl 0.076$ & \multirow{2}{*}{$0$} & \multirow{2}{*}{$0$} & \multirow{2}{*}{$0$} \\ & $\Sigma$\Bstrut  & $3.57\text{e7}$ & $3.77\text{e7}$ & \hl $\hl 0.421$ \\ \hline 
\multirow{2}{*}{$\bigg(\begin{pmatrix} 0.3 \\ -0.2 \end{pmatrix},\begin{pmatrix} 0.2 & 0.1 \\ 0.1 & 0.4 \end{pmatrix}\bigg)$} & $\mu$\Tstrut  & $0.038$ & $0.374$ & \hl $\hl 0.032$ & \multirow{2}{*}{$0$} & \multirow{2}{*}{$0$} & \multirow{2}{*}{$0$} \\ & $\Sigma$\Bstrut  & $0.061$ & $579$ & \hl $\hl 0.046$ \\ \hline 
\multirow{2}{*}{$\bigg(\begin{pmatrix} 0.5 \\ 0.5 \end{pmatrix},\begin{pmatrix} 0.1 & 0.1 \\ 0.1 & 0.8 \end{pmatrix}\bigg)$} & $\mu$\Tstrut  & $0.119$ & $0.111$ & \hl $\hl 0.086$ & \multirow{2}{*}{$0$} & \multirow{2}{*}{$0$} & \multirow{2}{*}{$0$} \\ & $\Sigma$\Bstrut  & $0.202$ & $0.261$ & \hl $\hl 0.139$ \\ \hline 
\end{tabular} 
\caption{\label{mult_trunc_normal_sim} Simulation results for the $TN(\mu,\Sigma)$ distribution for $n=1000$ and $10{,}000$ repetitions.}
\end{table}

\section{Products of independent distributions} \label{section_product_distr}
We consider truncated products of independent probability distributions. Let $p_{\theta^{(i)}}^{(i)} , i=1,\ldots,d$, be the smooth differentiable densities of $d$ probability distributions $\mathbb{P}_{\theta^{(1)}}^{(1)},\ldots, \mathbb{P}_{\theta^{(d)}}^{(d)}$. Each distribution depends on a parameter $\theta^{(i)} \in \Theta^{(i)} \subset \mathbb{R}^{p_i}$ and is defined on an interval $(a_i,b_i)$, where $-\infty \leq a_i < b_i \leq \infty$. Then, the multivariate density of $\mathbb{P}_{\theta^{(1)}}^{(1)} \otimes \ldots \otimes \mathbb{P}_{\theta^{(d)}}^{(d)} $ truncated with respect to a domain $K \subset  (a_1,b_1) \times \ldots \times (a_d,b_d)=:(a,b)$ is given by
\begin{align*}
    p_{\theta}(x) = \frac{1}{C(\theta)} \prod_{i=1}^d p_{\theta^{(i)}}^{(i)}(x_i), \qquad x=(x_1, \ldots, x_d) \in K,
\end{align*}
with $\theta = (\theta^{(1)},\ldots, \theta^{(d)}) \in \Theta \subset \mathbb{R}^p $, where $p=p_1+\ldots +p_d$ and the normalising constant is given by $C(\theta)=\int_K \prod_{i=1}^d p_{\theta^{(i)}}^{(i)}(x_i)\,dx$. Our objective is to estimate the parameter $\theta$. In the untruncated case, this is rather straightforward, as the parameters $\theta^{(i)}$ can be estimated in each direction separately, assuming that convenient estimation techniques exist for each probability distribution $\mathbb{P}_{\theta^{(i)}}^{(i)}$. However, things become more complicated if we restrict the domain of the product distribution to a subset $K$. In particular, estimation becomes challenging if $K$ is not itself a cube $(k_1^{(-)},k_1^{(+)}) \times \ldots \times (k_d^{(-)},k_d^{(+)})$ with $(k_i^{(-)},k_i^{(+)}) \subset (a_i,b_i)$, $i=1,\ldots ,d$, as in this case the truncated distribution is no longer a product distribution and therefore, parameter estimation cannot be performed separately in each dimension. However, Stein operators can be used in a similar way as in Section \ref{section_mult_trunc_normal} to obtain simple estimators even for complicated truncation domains. Our proposed estimation method works well if suitable density Stein operators are available for all marginal distributions $\mathbb{P}_{\theta^{(i)}}^{(i)}$, $i=1,\ldots,d$, as one has for a differentiable function $f$ that
\begin{align} \label{dens_op_trunc}
    \frac{\nabla \big( p_{\theta^{(1)}}^{(1)}(x_1) \ldots p_{\theta^{(d)}}^{(d)}(x_d) f(x) \big)}{p_{\theta^{(1)}}^{(1)}(x_1) \ldots p_{\theta^{(d)}}^{(d)}(x_d)} = \bigg(\frac{\frac{\partial}{\partial x_1}p_{\theta^{(1)}}^{(1)}(x_1)}{p_{\theta^{(1)}}^{(1)}(x_1)} f(x) + \frac{\partial}{\partial x_1} f(x), \ldots, \frac{\frac{\partial}{\partial x_d} p_{\theta^{(d)}}^{(d)}(x_d)}{p_{\theta^{(d)}}^{(d)}(x_d)} f(x) + \frac{\partial}{\partial x_d} f(x) \bigg)^{\top}.
\end{align}
 We refer to \cite{ebner2023point} where parameter estimators based on the density approach Stein operator for univariate probability distributions have been worked out. Note that it is possible to add a suitable function $\tau_{\theta}$ in the numerator on the left-hand side of \eqref{dens_op_trunc} in order to simplify the resulting operator (for example, the product of the Stein kernels of the marginal distributions, see \cite{ebner2023point,ley2017stein}). We then define the Stein operator for $p_{\theta}$ by
\begin{align} \label{stein_op_productdistr}
    \mathcal{A}_{\theta}f(x) =  \frac{\nabla \big( p_{\theta}(x) \tau_{\theta}(x) f(x) \big)}{p_{\theta}(x)} .
\end{align}
Let $R^{(1)}=\partial \overline{K} \cap \partial \overline{(a,b)}$ and $R^{(2)}=  \partial \overline{K} \setminus R^{(1)}$. We then have the following theorem, whose proof is similar to the one of Theorem \ref{theorem_trunc_norm_stein_op}.
\begin{Theorem} \label{theorem_productdistr_steinop_zero}
    Suppose that Assumption \ref{ass_piece_smooth_domain} holds, and let $f,\tau_{\theta} \in C^{\infty}(\mathrm{int}(K),\mathbb{R}) \cap C(\overline{K},\mathbb{R})$ be such that $f(x) = 0$ for $x \in R^{(2)}$ as well as $f(x) p_{\theta}(x) \tau_{\theta}(x) \Vert x \Vert^{d-1} \rightarrow 0 $ if $\Vert x \Vert \rightarrow \infty$ or if $x \rightarrow R^{(1)}$. Moreover, suppose that $\int_K \Vert \nabla ( p_{\theta}(x) \tau_{\theta}(x) f(x) ) \Vert\, dx < \infty $. Then we have that
    \begin{align*}
        \mathbb{E} [\mathcal{A}_{\theta} f(X)] = 0,
    \end{align*}
    where $X$ is a random variable with pdf $p_{\theta}$. 
\end{Theorem}

In this section, we look at two concrete examples to illustrate the approach in concrete terms and to allow comparisons to existing methods: A product of a normal and a gamma distribution and a product of a normal and a beta distribution, whereby we restrict ourselves to circles regarding the truncation domain. We refer to \cite{ebner2023point} in which the authors derived the Stein estimators for the corresponding univariate distributions. Let us consider the first example which is a product of independent $N(\mu, \sigma^2)$ and $\Gamma(\alpha,\beta)$ distributions. The product distribution therefore has the joint density $p(x_1,x_2)= p_{\theta^{(1)}}^{(1)}(x_1)p_{\theta^{(2)}}^{(2)}(x_2)/C(\theta)$, where
\begin{align*}
    p_{\theta^{(1)}}^{(1)}(x_1)= \frac{1}{\sqrt{2\pi\sigma^2}}\exp\bigg( -\frac{(x_1-\mu)^2}{2\sigma^2} \bigg) , \qquad p_{\theta^{(2)}}^{(2)}(x_2)= \frac{\beta^{\alpha}}{\Gamma(\alpha)} x_2^{\alpha-1} e^{-\beta x_2}, \qquad x_1 \in \mathbb{R},\: x_2>0,
\end{align*}
with $\theta^{(1)}=(\mu, \sigma^2), \, \theta^{(2)}=(\alpha, \beta)$, and therefore $\theta=(\mu, \sigma^2, \alpha, \beta)$ and $C(\theta)= \int_{K}p_{\theta^{(1)}}^{(1)}(x_1)p_{\theta^{(2)}}^{(2)}(x_2)\,dx_1\,dx_2 $. With the choice $\tau_{\theta}(x)=x_2$ the Stein operator \eqref{stein_op_productdistr} reads
\begin{align*}
    \mathcal{A}_{\theta} f(x) = \frac{\nabla \big(p_{\theta^{(1)}}^{(1)}(x_1)p_{\theta^{(2)}}^{(2)}(x_2) x_2 f(x) \big) }{p_{\theta^{(1)}}^{(1)}(x_1)p_{\theta^{(2)}}^{(2)}(x_2)} = \begin{pmatrix} \frac{x_2(\mu -x_1)}{\sigma^2} f(x) + x_2 \frac{\partial}{\partial x_1} f(x) \\[5pt] (\alpha - \beta x_2) f(x) + x_2 \frac{\partial}{\partial x_2} f(x) \end{pmatrix}.
\end{align*}
Here we suppose that the truncation domain is a (possibly truncated) circle $B_r^2(m)$ such that $K= B_r^2(m) \cap \mathbb{R} \times (0,\infty) \neq \emptyset $. In the sequel, we will write $\mathbb{Q}_{\theta}$ for the product distribution of $N(\mu, \sigma^2)$ and $\Gamma(\alpha,\beta)$ truncated with respect to the set $K$. Similarly to Section \ref{section_mult_trunc_normal}, we let $\kappa(x)= (x_1 - m_1)^2 + (x_2-m_2)^2 - r^2$ and choose two test functions $f_1:\overline{K} \rightarrow \mathbb{R}$, $x \mapsto \kappa(x) $ and $f_2:\overline{K} \rightarrow \mathbb{R}$, $x \mapsto \kappa(x)(x_1+x_2) $. With Theorem \ref{theorem_productdistr_steinop_zero} we have
\begin{align*}
    \mathbb{E} [\mathcal{A}_{\theta} f(X)] = 0 
\end{align*}
for $f=f_1$ or $f=f_2$ if $X \sim \mathbb{Q}_{\theta}$ for all $\theta \in \Theta$. \par
We now let $X_1, \ldots, X_n \sim \mathbb{Q}_{\theta_0}$ (where $X_i = (X_i^{(1)},X_i^{(2)})$) be i.i.d.\ random variables defined on a common probability space $(\Omega,\mathcal{F},\mathbb{P})$. The estimator for $\theta$ is obtained by solving
\begin{align*}
    \frac{1}{n} \sum_{i=1}^n \mathcal{A}_{\theta} f_j(X_i) = 0, \qquad j=1,2,
\end{align*}
for $\theta$, which gives
\begin{align*}
    \hat{\mu}_n&=\frac{\overline{X^{(2)}\frac{\partial}{\partial x_1} f_2(X) } \ \overline{X^{(2)}X^{(1)}f_1(X)}-\overline{X^{(2)}\frac{\partial}{\partial x_1}f_1(X)} \ \overline{X^{(2)}X^{(1)}f_2(X)}}{\overline{X^{(2)}f_1(X)} \ \overline{X^{(2)}\frac{\partial}{\partial x_1}f_2(X)}- \overline{X^{(2)}\frac{\partial}{\partial x_1}f_1(X)} \ \overline{X^{(2)}f_2(X)}}, \\
    \hat{\sigma}_n^2&=\frac{\overline{X^{(2)}f_1(X)} \ \overline{X^{(2)}X^{(1)}f_2(X)} -\overline{X^{(2)}f_2(X)} \ \overline{X^{(2)}X^{(1)}f_1(X)}}{\overline{X^{(2)}f_1(X)} \ \overline{X^{(2)}\frac{\partial}{\partial x_1}f_2(X)}- \overline{X^{(2)}\frac{\partial}{\partial x_1}f_1(X)} \ \overline{X^{(2)}f_2(X)}}, \\
    \hat{\alpha}_n&=\frac{\overline{X^{(2)}f_2(X)} \ \overline{X^{(2)}\frac{\partial}{\partial x_2}f_1(X)}-\overline{X^{(2)}f_1(X)} \ \overline{X^{(2)}\frac{\partial}{\partial x_2}f_2(X)}}{\overline{X^{(2)}f_1(X)} \ \overline{f_2(X)}-\overline{f_1(X)} \ \overline{X^{(2)}f_2(X)}}, \\
    \hat{\beta}_n&=\frac{\overline{f_2(X)} \ \overline{X^{(2)}\frac{\partial}{\partial x_2}f_1(X)}-\overline{f_1(X)} \ \overline{X^{(2)}\frac{\partial}{\partial x_2}f_2(X)}}{\overline{X^{(2)}f_1(X)} \ \overline{f_2(X)}-\overline{f_1(X)} \ \overline{X^{(2)}f_2(X)}}.
\end{align*}
Consistency and asymptotic normality can be worked out with standard procedures for moment estimation as in Section \ref{section_mult_trunc_normal}. We compared the Stein estimator $\hat{\theta}_n^{\mathrm{ST}}=(\hat{\mu}_n,\hat{\sigma}_n^2,\hat{\alpha}_n,\hat{\beta}_n)$ to the MLE $\hat{\theta}_n^{\mathrm{ML}}$ and the score matching approach $\hat{\theta}_n^{\mathrm{SM}}$ by means of a competitive simulation study. The MLE is calculated via numerical optimisation of the log-likelihood function. We used the optimisation algorithm \textit{L-BFGS-B} as implemented in the R function \texttt{optim} since it allows for box constraints which are needed for the parameters $\sigma^2, \alpha$ and $\beta$. The point $(0,1,1,1)$ was used as an initial guess for the optimisation algorithm. Note that $\hat{\theta}_n^{\mathrm{SM}}$ is explicit here and does not require numerical optimisation. As in Section \ref{section_mult_trunc_normal}, we added a column \textit{NE} to report the estimated relative frequency of non-eligible estimates. Here this is the case if the estimator returned negative values for $\sigma^2, \alpha$ or $\beta$, or if the optimisation procedure for the MLE threw an error (e.g.\ because it did not converge). The simulation results can be found in Table \ref{mult_trunc_normalgamma_sim}. As one can observe, the score matching approach yields overall the best results. The Stein estimator performs well in comparison to the MLE as the latter has tremendous difficulties regarding convergence of the algorithm. For the parameter constellation $(0,0.1,0.5,3)$ the MLE algorithm did not converge a single time out of the $10{,}000$ Monte Carlo repetitions and also for all other parameter values, a significant part of the estimates could not be calculated. Note that bias and MSE were calculated with respect to the Monte Carlo repetitions where the estimate was eligible. This means that one has to be careful with comparing the bias and MSE of the MLE as only the estimates for which the optimisation algorithm was converging are included in the simulation. All three estimators seem to have difficulties to estimate the parameters of the normal distribution if $\sigma^2$ is large, which seems natural. 

\begin{table}[h]
\centering
\begin{tabular}{cc|ccc|ccc|ccc}
 $\theta_0$ & & \multicolumn{3}{c|}{Bias} & \multicolumn{3}{c|}{MSE} & \multicolumn{3}{c}{NE} \\ \hline
   & & $\hat{\theta}_n^{\mathrm{ML}}$ & $\hat{\theta}_n^{\mathrm{SM}}$ & $\hat{\theta}_n^{\mathrm{ST}}$ & $\hat{\theta}_n^{\mathrm{ML}}$ & $\hat{\theta}_n^{\mathrm{SM}}$ & $\hat{\theta}_n^{\mathrm{ST}}$ & $\hat{\theta}_n^{\mathrm{ML}}$ & $\hat{\theta}_n^{\mathrm{SM}}$  & $\hat{\theta}_n^{\mathrm{ST}}$ \\ \hline
\multirow{4}{*}{$ \begin{pmatrix} 1 \\ 2 \\ 3 \\ 4 \end{pmatrix}$} & $\mu$\Tstrut  & $6.29$ & \hl $\hl 1.39$ & $2.46$ & $8207$ & \hl $\hl 193$ & $3910$ & \multirow{4}{*}{$20$} & \multirow{4}{*}{$0$} & \multirow{4}{*}{$0$} \\ & $\sigma^2$  & $10.7$ & \hl $\hl 2.9$ & $5.94$ & $3.32\text{e4}$ & \hl $\hl 745$ & $2.51\text{e4}$ \\ & $\alpha$  & \hl $\hl 0.09$ & $0.348$ & $0.579$ & \hl $\hl 1.95$ & $2.98$ & $4.09$ \\ & $\beta$\Bstrut  & \hl $\hl 0.052$ & $0.214$ & $0.354$ & \hl $\hl 0.794$ & $1.25$ & $1.65$ \\ \hline 
\multirow{4}{*}{$\begin{pmatrix} 0.5 \\ 1 \\ 4 \\ 5 \end{pmatrix}$} & $\mu$\Tstrut  & \hl $\hl0.347$ & $0.529$ & $0.503$ & \hl $\hl 2.33$ & $149$ & $118$ & \multirow{4}{*}{$27$} & \multirow{4}{*}{$0$} & \multirow{4}{*}{$0$} \\ & $\sigma^2$  & \hl $\hl 0.787$ & $1.18$ & $1.3$ & \hl $\hl 12.6$ & $749$ & $1803$ \\ & $\alpha$  & $-0.251$ & \hl $\hl 0.221$ & $0.373$ & \hl $\hl 2.67$ & $3.79$ & $5.07$ \\ & $\beta$\Bstrut  & $-0.165$ & \hl $\hl 0.152$ & $0.247$ & \hl $\hl 1.11$ & $1.68$ & $2.16$ \\ \hline 
\multirow{4}{*}{$\begin{pmatrix} 0 \\ 1 \\ 1 \\ 1 \end{pmatrix}$} & $\mu$\Tstrut  & $6.26\text{e-3}$ & \hl $\hl -4.36\text{e-4}$ & $-3.08\text{e-3}$ & $0.58$ & \hl $\hl 0.062$ & $0.163$ & \multirow{4}{*}{$28$} & \multirow{4}{*}{$0$} & \multirow{4}{*}{$0$} \\ & $\sigma^2$  & $1.1$ &\hl $\hl 0.309$ & $0.477$ & $826$ & \hl $\hl 8.45$ & $314$ \\ & $\alpha$  & \hl $\hl 0.517$ & $0.557$ & $0.73$ & \hl $\hl 1.13$ & $1.39$ & $1.96$ \\ & $\beta$\Bstrut  & \hl $\hl 0.289$ & $0.316$ & $0.404$ & \hl $\hl 0.359$ & $0.457$ & $0.608$ \\ \hline 
\multirow{4}{*}{$\begin{pmatrix} 0 \\ 0.1 \\ 0.5 \\ 3 \end{pmatrix}$} & $\mu$\Tstrut  & $-$ & \hl $\hl 4.68\text{e-4}$ & $-2.8\text{e-3}$ & $-$ & $3.84\text{e-4}$ & \hl $\hl 3.82\text{e-4}$ & \multirow{4}{*}{$100$} & \multirow{4}{*}{$0$} & \multirow{4}{*}{$0$} \\ & $\sigma^2$  & $-$ & \hl $\hl -5.96\text{e-4}$ & $-2.82\text{e-3}$ & $-$ & $8.11\text{e-3}$ & \hl $\hl 7.8\text{e-3}$ \\ & $\alpha$  & $-$ & \hl $\hl 1.26$ & $1.38$ & $-$ & \hl $\hl 3.31$ & $3.95$ \\ & $\beta$\Bstrut  & $-$ & \hl $\hl 0.856$ & $0.915$ & $-$ & \hl $\hl 1.56$ & $1.79$ \\ \hline 
\multirow{4}{*}{$\begin{pmatrix} 0.2 \\ 0.3 \\ 0.1 \\ 1 \end{pmatrix}$} & $\mu$\Tstrut  & $0.024$ & \hl $\hl -2.8\text{e-3}$ & $-8.44\text{e-3}$ & $2.84\text{e-3}$ & \hl $\hl 1.88\text{e-3}$ & $2.06\text{e-3}$ & \multirow{4}{*}{$99$} & \multirow{4}{*}{$0$} & \multirow{4}{*}{$0$} \\ & $\sigma^2$  & $0.285$ & \hl $\hl -2.09\text{e-3}$ & $-0.011$ & $0.085$ & $0.046$ & \hl $\hl 0.043$ \\ & $\alpha$  & $1.1$ & \hl $\hl 1.05$ & $1.18$ & \hl $\hl 1.83$ & $1.85$ & $2.3$ \\ & $\beta$\Bstrut  & $0.675$ & \hl $\hl 0.621$ & $0.682$ & $0.669$ & \hl $\hl 0.661$ & $0.78$ \\ \hline 
\multirow{4}{*}{$\begin{pmatrix} 0 \\ 1.5 \\ 3 \\ 0.5 \end{pmatrix}$} & $\mu$\Tstrut  & $-0.294$ & $0.053$ & \hl $\hl 0.038$ & $3982$ & \hl $\hl 9.27$ & $13.8$ & \multirow{4}{*}{$28$} & \multirow{4}{*}{$0$} & \multirow{4}{*}{$0$} \\ & $\sigma^2$  & $15.4$ & \hl $\hl 1.27$ & $2.18$ & $2.05\text{e5}$ & \hl $\hl 339$ & $7982$ \\ & $\alpha$  & $0.725$ & \hl $\hl 0.718$ & $1.09$ & \hl $\hl 1.72$ & $1.84$ & $3.37$ \\ & $\beta$\Bstrut  & $0.372$ & \hl $\hl 0.366$ & $0.54$ & \hl $\hl 0.442$ & $0.476$ & $0.821$ \\ \hline 
\multirow{4}{*}{$\begin{pmatrix} 0 \\ 0.4 \\ 3 \\ 3 \end{pmatrix}$} & $\mu$\Tstrut  & $-1.32\text{e-4}$ & $-5.59\text{e-4}$ & \hl $\hl 7.58\text{e-6}$ & $2.68\text{e-3}$ & $2.53\text{e-3}$ & \hl $\hl 2.33\text{e-3}$ & \multirow{4}{*}{$64$} & \multirow{4}{*}{$0$} & \multirow{4}{*}{$0$} \\ & $\sigma^2$  & $0.325$ & \hl $\hl 0.013$ & $0.016$ & $0.113$ & \hl $\hl 0.073$ & $0.078$ \\ & $\alpha$  & \hl $\hl 0.093$ & $0.139$ & $0.224$ & \hl $\hl 1.78$ & $2.08$ & $2.6$ \\ & $\beta$\Bstrut  & \hl $\hl 0.074$ & $0.088$ & $0.135$ & \hl $\hl 0.64$ & $0.774$ & $0.923$ \\ \hline 
\end{tabular} 
\caption{\label{mult_trunc_normalgamma_sim} Simulation results for the product of $N(\mu, \sigma^2)$ and $\Gamma(\alpha, \beta)$ for $n=500$ and $10{,}000$ repetitions. The truncation domain is the circle with $m=(0,2)$ and $r=1$.}
\end{table}

Let us consider the second example which is a product of the normal distribution $N(\mu, \sigma^2)$ and the beta distribution $\mathrm{Beta}(\alpha, \beta)$. We therefore have
\begin{align*}
    p_{\theta^{(1)}}^{(1)}(x_1)= \frac{1}{\sqrt{2\pi\sigma^2}}\exp\bigg( -\frac{(x_1-\mu)^2}{2\sigma^2} \bigg) , \qquad p_{\theta^{(2)}}^{(2)}(x_2)= \frac{x_2^{\alpha-1} (1-x_2)^{\beta-1}}{B(\alpha,\beta)} , \qquad x_1 \in \mathbb{R},\: 0<x_2<1,
\end{align*}
with $\theta^{(1)}=(\mu, \sigma^2), \, \theta^{(2)}=(\alpha, \beta)$ and $\theta=(\mu, \sigma^2, \alpha, \beta)$, and the beta function is given by $B(\alpha,\beta)=\Gamma(\alpha)\Gamma(\beta)/\Gamma(\alpha+\beta)$. The truncation domain $K$ is again a circle ball of radius $r$ and center $m=(m_1,m_2)$ such that $K = B_r^2(m) \cap \mathbb{R} \times [0,1] \neq \emptyset$. We then define a Stein operator with $\tau_{\theta}(x)=x_2(1-x_2)$ by
\begin{align*}
    \mathcal{A}_{\theta} f(x) = \frac{\nabla \big(p_{\theta^{(1)}}^{(1)}(x_1)p_{\theta^{(2)}}^{(2)}(x_2) x_2(1-x_2) f(x) \big) }{p_{\theta^{(1)}}^{(1)}(x_1)p_{\theta^{(2)}}^{(2)}(x_2)} = \begin{pmatrix} \frac{x_2(1-x_2)(\mu -x_1)}{\sigma^2} f(x) + x_2(1-x_2) \frac{\partial}{\partial x_1} f(x) \\[5pt] (\alpha -(\alpha + \beta) x_2) f(x) + x_2(1-x_2) \frac{\partial}{\partial x_2} f(x) \end{pmatrix}.
\end{align*}
Again, we write $\mathbb{Q}_{\theta}$ for the product distribution of $N(\mu,\sigma^2)$ and $\mathrm{Beta}(\alpha, \beta)$ truncated with respect to $K$ and let $X_1, \ldots, X_n \sim \mathbb{Q}_{\theta_0}$ (where $X_i = (X_i^{(1)},X_i^{(2)})$) be i.i.d.\ random variables defined on a common probability space $(\Omega,\mathcal{F},\mathbb{P})$. With the exact same test functions $f_1,f_2$ as in the previous example we solve
\begin{align*}
    \frac{1}{n} \sum_{i=1}^n \mathcal{A}_{\theta} f_j(X_i) = 0, \qquad j=1,2,
\end{align*}
for $\theta$, which gives
\begin{alignat*}{2}
    \hat{\mu}_n&= \frac{M_n^{(1)}M_n^{(2)}-M_n^{(3)}M_n^{(4)}}{M_n^{(5)}M_n^{(1)}-M_n^{(3)}M_n^{(6)}}, \qquad &&\hat{\sigma}_n^2= \frac{M_n^{(5)}M_n^{(4)}-M_n^{(6)}M_n^{(2)}}{M_n^{(5)}M_n^{(1)}-M_n^{(3)}M_n^{(6)}} \\
    \hat{\alpha}_n&=\frac{O_n^{(1)}O_n^{(2)}-O_n^{(3)}O_n^{(4)}}{O_n^{(5)}O_n^{(1)}-O_n^{(3)}O_n^{(6)}}, \qquad &&\hat{\beta}_n=\frac{O_n^{(5)}O_n^{(4)}-O_n^{(6)}O_n^{(2)}}{O_n^{(5)}O_n^{(1)}-O_n^{(3)}O_n^{(6)}},
\end{alignat*}
where
\begin{alignat*}{2}
    M_n^{(1)}=& \overline{(1-X^{(2)})X^{(2)}\tfrac{\partial}{\partial x_1} f_2(X) }, \qquad &&M_n^{(2)}=  \overline{(1-X^{(2)})X^{(2)}X^{(1)}f_1(X)}, \\
    M_n^{(3)}=& \overline{(1-X^{(2)})X^{(2)}\tfrac{\partial}{\partial x_1}f_1(X)} , \qquad &&M_n^{(4)}= \overline{(1-X^{(2)})X^{(2)}X^{(1)}f_2(X)}, \\
    M_n^{(5)}=& \overline{(1-X^{(2)})X^{(2)}f_1(X)}, \qquad &&M_n^{(6)}= \overline{(1-X^{(2)})X^{(2)}f_2(X)}, \\
    O_n^{(1)}=& \overline{X^{(2)}f_1(X)}, \qquad &&O_n^{(2)}=\overline{(1-X^{(2)})X^{(2)}\tfrac{\partial}{\partial x_2}f_2(X)}, \\
    O_n^{(3)}=& \overline{X^{(2)}f_2(X)} , \qquad &&O_n^{(4)}= \overline{(1-X^{(2)})X^{(2)}\tfrac{\partial}{\partial x_2}f_1(X)}, \\
    O_n^{(5)}=& \overline{(X^{(2)}-1)f_2(X)} , \qquad &&O_n^{(6)}= \overline{(X^{(2)}-1)f_1(X)}.
\end{alignat*}
The results of the simulation study are available in Table \ref{mult_trunc_normalbeta_sim}. We compared the Stein estimator $\hat{\theta}_n^{\mathrm{ST}}=(\hat{\mu}_n,\hat{\sigma}_n^2,\hat{\alpha}_n,\hat{\beta}_n)$ to the MLE $\hat{\theta}_n^{\mathrm{ML}}$ and the score matching estimator $\hat{\theta}_n^{\mathrm{SM}}$. The procedure to compute the MLE is exactly the same as for the previous example and $\hat{\theta}_n^{\mathrm{SM}}$ can be worked out explicitly as before. We can observe in the column \textit{NE} that the MLE has severe difficulties regarding the computation of the estimates: The optimisation algorithm often does not converge. Nonetheless, the Stein and score matching estimators returned eligible values for all Monte Carlo repetitions. As per bias and MSE, the table reports sometimes lower values for $\hat{\theta}_n^{\mathrm{SM}}$, sometimes for $\hat{\theta}_n^{\mathrm{ST}}$, depending on the true parameter values. However, bias and MSE for the MLE have to treated carefully since these statistics do not take into account the Monte Carlo repetitions for which the estimator did not exist. Similarly to the product of a normal and a gamma distribution, all estimators have difficulties for large $\sigma^2$. Overall, we recommend to use the Stein estimator or the score matching estimator.

\begin{table}[h] 
\centering
\begin{tabular}{cc|ccc|ccc|ccc}
 $\theta_0$ & & \multicolumn{3}{c}{Bias} & \multicolumn{3}{|c}{MSE} & \multicolumn{3}{|c}{NE} \\  \hline
   & & $\hat{\theta}_n^{\mathrm{ML}}$ & $\hat{\theta}_n^{\mathrm{SM}}$ & $\hat{\theta}_n^{\mathrm{ST}}$ & $\hat{\theta}_n^{\mathrm{ML}}$ & $\hat{\theta}_n^{\mathrm{SM}}$ & $\hat{\theta}_n^{\mathrm{ST}}$ & $\hat{\theta}_n^{\mathrm{ML}}$ & $\hat{\theta}_n^{\mathrm{SM}}$  & $\hat{\theta}_n^{\mathrm{ST}}$ \\ \hline
\multirow{4}{*}{$\begin{pmatrix} 1 \\ 2 \\ 1 \\ 1 \end{pmatrix}$} & $\mu$\Tstrut  & $60.9$ & $1.22$ & \hl $\hl 0.797$ & $1.38\text{e5}$ & $1149$ & \hl $\hl 1058$ & \multirow{4}{*}{$34$} & \multirow{4}{*}{$0$} & \multirow{4}{*}{$0$} \\ & $\sigma^2$  & $148$ & $2.4$ & \hl $\hl 2$ & $1.02\text{e6}$ & \hl $\hl 5035$ & $8852$ \\ & $\alpha$  & \hl $\hl 0.015$ & $0.054$ & $0.063$ & \hl $\hl 8.44\text{e-3}$ & $0.028$ & $0.034$ \\ & $\beta$\Bstrut  & \hl $\hl 0.015$ & $0.056$ & $0.06$ & \hl $\hl 8.64\text{e-3}$ & $0.028$ & $0.034$ \\ \hline 
\multirow{4}{*}{$\begin{pmatrix} 0.5 \\ 0.1 \\ 4 \\ 5 \end{pmatrix}$} & $\mu$\Tstrut  & $-$ & \hl $\hl 0.014$ & $0.017$ & $-$ & \hl $\hl 9.88\text{e-3}$ & $0.011$ & \multirow{4}{*}{$100$} & \multirow{4}{*}{$0$} & \multirow{4}{*}{$0$} \\ & $\sigma^2$  & $-$ & \hl $\hl 3.9\text{e-3}$ & $4.64\text{e-3}$ & $-$ & \hl $\hl 9.56\text{e-3}$ & $9.8\text{e-3}$ \\ & $\alpha$  & $-$ & $0.042$ & \hl $\hl 0.034$ & $-$ & \hl $\hl 0.103$ & $0.202$ \\ & $\beta$\Bstrut  & $-$ & $0.05$ & \hl $\hl 0.043$ & $-$ & \hl $\hl 0.157$ & $0.267$ \\ \hline 
\multirow{4}{*}{$\begin{pmatrix} 0 \\ 1 \\ 1 \\ 1.5 \end{pmatrix}$} & $\mu$\Tstrut  & $-6.25$ & $-0.209$ & \hl $\hl -0.041$ & $1\text{e5}$ & $231$ & \hl $\hl 42.9$ & \multirow{4}{*}{$28$} & \multirow{4}{*}{$0$} & \multirow{4}{*}{$0$} \\ & $\sigma^2$  & $209$ & $2.28$ & \hl $\hl 1.89$ & $3.45\text{e6}$ & $4518$ & \hl $\hl 2004$ \\ & $\alpha$  & \hl $\hl 0.01$ & $0.047$ & $0.052$ & \hl $\hl 8.18\text{e-3}$ & $0.027$ & $0.033$ \\ & $\beta$\Bstrut  & \hl $\hl 0.014$ & $0.049$ & $0.06$ & \hl $\hl 0.016$ & $0.031$ & $0.049$ \\ \hline 
\multirow{4}{*}{$\begin{pmatrix} 0 \\ 0.1 \\ 0.5 \\ 3 \end{pmatrix}$} & $\mu$\Tstrut  & $-3.72\text{e-3}$ & \hl $\hl -3.5\text{e-5}$ & $4.84\text{e-4}$ & \hl $\hl 4.25\text{e-4}$ & $7.66\text{e-4}$ & $8.17\text{e-4}$ & \multirow{4}{*}{$99$} & \multirow{4}{*}{$0$} & \multirow{4}{*}{$0$} \\ & $\sigma^2$  & $0.121$ & \hl $\hl 4.28\text{e-3}$ & $5.25\text{e-3}$ & $0.019$ & \hl $\hl 9.68\text{e-3}$ & $0.01$ \\ & $\alpha$  & \hl $\hl -0.01$ & $0.159$ & $0.022$ & \hl $\hl 3.21\text{e-3}$ & $0.061$ & $0.029$ \\ & $\beta$\Bstrut  & \hl $\hl 0.038$ & $0.391$ & $0.059$ & \hl $\hl 0.067$ & $0.402$ & $0.197$ \\ \hline 
\multirow{4}{*}{$\begin{pmatrix} 0.2 \\ 0.3 \\ 0.1 \\ 0.4 \end{pmatrix}$} & $\mu$\Tstrut  & $20.7$ & $0.271$ & \hl $\hl 0.131$ & $2.12\text{e4}$ & $25.2$ & \hl $\hl 3.53$ & \multirow{4}{*}{$97$} & \multirow{4}{*}{$0$} & \multirow{4}{*}{$0$} \\ & $\sigma^2$  & $37.3$ & $0.418$ & \hl $\hl 0.239$ & $7.02\text{e4}$ & $46.4$ & \hl $\hl 14.5$ \\ & $\alpha$  & \hl $\hl -0.031$ & $0.504$ & $0.044$ & \hl $\hl 1.81\text{e-3}$ & $0.313$ & $0.011$ \\ & $\beta$\Bstrut  & \hl $\hl 0.012$ & $0.29$ & $0.047$ & \hl $\hl 4.69\text{e-3}$ & $0.126$ & $0.017$ \\ \hline 
\multirow{4}{*}{$\begin{pmatrix} 0 \\ 1.5 \\ 1 \\ 0.5 \end{pmatrix}$} & $\mu$\Tstrut  & $12.5$ & \hl $\hl 0.141$ & $-1.08$ & $5.62\text{e5}$ & \hl $\hl 38.4$ & $2.13\text{e4}$ & \multirow{4}{*}{$23$} & \multirow{4}{*}{$0$} & \multirow{4}{*}{$0$} \\ & $\sigma^2$  & $1709$ & \hl $\hl 1.39$ & $14.2$ & $2.87\text{e7}$ & \hl $\hl 828$ & $6.04\text{e5}$ \\ & $\alpha$  & \hl $\hl 6.8\text{e-3}$ & $0.127$ & $0.061$ & \hl $\hl 8.76\text{e-3}$ & $0.049$ & $0.037$ \\ & $\beta$\Bstrut  & \hl $\hl 3.85\text{e-3}$ & $0.161$ & $0.052$ & \hl $\hl 3.61\text{e-3}$ & $0.059$ & $0.022$ \\ \hline 
\multirow{4}{*}{$\begin{pmatrix} 0 \\ 0.4 \\ 2 \\ 2 \end{pmatrix}$} & $\mu$\Tstrut  & $-2.48$ & $0.026$ & \hl $\hl -0.025$ & $1.26\text{e5}$ & \hl $\hl 2.08$ & $4.4$ & \multirow{4}{*}{$54$} & \multirow{4}{*}{$0$} & \multirow{4}{*}{$0$} \\ & $\sigma^2$  & $678$ & \hl $\hl 0.705$ & $1.2$ & $1.18\text{e7}$ & \hl $\hl 211$ & $1166$ \\ & $\alpha$  & \hl $\hl -0.018$ & $0.027$ & $0.025$ & \hl $\hl 0.023$ & $0.03$ & $0.062$ \\ & $\beta$\Bstrut  & \hl $\hl -0.019$ & $0.028$ & $0.027$ & \hl $\hl 0.022$ & $0.03$ & $0.062$ \\ \hline 
\end{tabular} 
\caption{\label{mult_trunc_normalbeta_sim} Simulation results for the product of $N(\mu, \sigma^2)$ and $\mathrm{Beta}(\alpha, \beta)$ for $n=500$ and $10{,}000$ repetitions. The truncation domain is the circle with $m=(0,0.5)$ and $r=0.5$.}
\end{table}

\section*{Acknowledgements}
AF is funded in part by ARC Consolidator grant from ULB and FNRS Grant CDR/OL J.0197.20 as well as EPSRC Grant EP/T018445/1. RG is funded in part by EPSRC grant EP/Y008650/1. YS is funded in part by ARC Consolidator grant from ULB and FNRS Grant CDR/OL J.0197.20.

\bibliography{library}
\bibliographystyle{abbrv}

\end{document}